\documentclass{amsart}

\usepackage[utf8]{inputenc}
\usepackage{graphicx}
\usepackage{mathtools}
\usepackage{tikz}
\usetikzlibrary{arrows, backgrounds, calc, chains, decorations, patterns, positioning, shapes}

\usepackage{amsfonts}
\usepackage{amsmath}
\usepackage{amsrefs}
\usepackage{amssymb}
\usepackage{amstext}
\usepackage{amsthm}

\usepackage[justification=centering]{caption}
\usepackage{hyperref}
\usepackage{geometry}

\newcommand{\area}{\mathsf{area}}
\newcommand{\dinv}{\mathsf{dinv}}
\newcommand{\shift}{\mathsf{shift}}

\newcommand{\D}{\mathsf{D}} 
\newcommand{\SQ}{\mathsf{SQ}} 
\newcommand{\LD}{\mathsf{LD}} 
\newcommand{\LSQ}{\mathsf{LSQ}} 

\DeclareFontFamily{U}{bigshuffle}{}
\DeclareFontShape{U}{bigshuffle}{m}{n}{
	<5-8> s*[1.7] shuffle7
	<8->  s*[1.7] shuffle10
}{}
\DeclareSymbolFont{BigShuffle}{U}{bigshuffle}{m}{n}
\DeclareMathSymbol\bigshuffle{\mathop}{BigShuffle}{"001}
\DeclareMathSymbol\bigcshuffle{\mathop}{BigShuffle}{"002}

\newcommand{\Ht}{\widetilde{H}}

\newcommand{\N}{\mathbb{N}}

\newcommand{\qbinom}[2]{\genfrac{[}{]}{0pt}{}{#1}{#2}}

\newcommand{\<}{\langle}
\renewcommand{\>}{\rangle}

\theoremstyle{plain}

\theoremstyle{definition}
\newtheorem{theorem}{Theorem}[section]
\newtheorem{conjecture}[theorem]{Conjecture}

\newtheorem{definition}[theorem]{Definition}
\newtheorem{lemma}[theorem]{Lemma}
\newtheorem{proposition}[theorem]{Proposition}

\theoremstyle{remark}
\newtheorem{remark}[theorem]{Remark}

\makeatletter%
\renewenvironment{proof}[1][\proofname]{%
	\par\pushQED{\qed}\normalfont%
	\topsep6\p@\@plus6\p@\relax
	\trivlist\item[\hskip\labelsep\bfseries#1\@addpunct{.}]%
	\ignorespaces
}{%
	\qedhere 
}
\makeatother

\makeatletter%
\DeclareRobustCommand*{\bfseries}{%
	\not@math@alphabet\bfseries\mathbf
	\fontseries\bfdefault\selectfont
	\boldmath
}
\makeatother

\title{Some consequences of the valley Delta conjectures}

\author{Michele D'Adderio}
\address
{Dipartimento di Matematica \newline \indent
Universit\`a di Pisa \newline \indent
Pisa, PI, 56127, Italia}
\email{michele.dadderio@unipi.it}

\author{Alessandro Iraci}
\address
{Laboratoire d'alg\`ebre, de combinatoire, et d'informatique math\'ematique \newline \indent
Universit\'e du Qu\'ebec \`a Montr\'eal \newline \indent
Montr\'eal, QC, H2X 3Y7, Canada}
\email{iraci.alessandro@uqam.ca}

\begin{document}

\begin{abstract}
In \cite{Haglund-Remmel-Wilson-2018} Haglund, Remmel and Wilson introduced their \emph{Delta conjectures}, which give two different combinatorial interpretations of the symmetric function $\Delta'_{e_{n-k-1}} e_n$ in terms of rise-decorated or valley-decorated labelled Dyck paths respectively. While the rise version has been recently proved \cites{DAdderio-Mellit-Compositional-Delta-2020,Blasiak-Haiman-Morse-Pun-Seeling-Extended-Delta-2021}, not much is known about the valley version. In this work we prove the Schr\"oder case of the valley Delta conjecture, the Schr\"oder case of its square version \cite{Iraci-VandenWyngaerd-Valley-Square-2021}, and the Catalan case of its extended version \cite{Qiu-Wilson-2020}. Furthermore, assuming the symmetry of (a refinement of) the combinatorial side of the extended valley Delta conjecture, we deduce also the Catalan case of its square version \cite{Iraci-VandenWyngaerd-Valley-Square-2021}.
\end{abstract}

\maketitle

\section{Introduction}

In \cite{Haglund-Remmel-Wilson-2018} the authors introduced their \emph{Delta conjectures}, which give two different combinatorial interpretations of the symmetric function $\Delta'_{e_{n-k-1}} e_n$ in terms of rise-decorated or valley-decorated labelled Dyck paths respectively. More precisely,
\begin{equation} \label{eq:DeltaConjs}
	\Delta'_{e_{n-k-1}} e_n=\sum_{\pi\in \LD(0,n)^{* k}}q^{\dinv(\pi)}t^{\area(\pi)}=\sum_{\pi\in \LD(0,n)^{\bullet k}}q^{\dinv(\pi)}t^{\area(\pi)}
\end{equation}
(see Sections~\ref{sec:CombDef}~and~\ref{sec:SF} for the missing definitions).

This symmetric function is of particular interest as it gives conjecturally the Frobenius characteristic of the so called \emph{super diagonal coinvariants} \cite{Zabrocki-Delta-Module-2019}.

The rise version has been extensively studied \cites{DAdderio-Iraci-VandenWyngaerd-Delta-t0-2020, DAdderio-Iraci-VandenWyngaerd-TheBible-2019, DAdderio-Iraci-VandenWyngaerd-GenDeltaSchroeder-2019, Garsia-Haglund-Remmel-Yoo-2019, Romero-Deltaq1-2017} before being finally proved: in \cite{DAdderio-Mellit-Compositional-Delta-2020} it is proved the compositional refinement introduced in \cite{DAdderio-Iraci-VandenWyngaerd-Theta-2021}, and in \cite{Blasiak-Haiman-Morse-Pun-Seeling-Extended-Delta-2021} it is proved the extended version. On the other hand, the valley version received significantly less attention, this fact being mainly due to technical difficulties. Before mentioning what is known about the valley Delta to this day, we want to remark the intrinsic interest in this version of the conjecture: indeed in \cite{Haglund-Sergel-2021} a conjectural basis of the super diagonal coinvariants in \cite{Zabrocki-Delta-Module-2019} is provided, that would explain the Hilbert series predicted by the valley Delta conjecture (not the rise Delta!).

To this date, extended \cite{Qiu-Wilson-2020} and square versions \cite{Iraci-VandenWyngaerd-Valley-Square-2021} of the valley Delta have been formulated. In \cites{Iraci-VandenWyngaerd-Valley-Square-2021,Iraci-VandenWyngaerd-pushing-2021} it is proved that the original valley Delta conjecture implies these other versions, a surprising fact that has no analogue for the rise Delta conjectures: indeed, the rise version of the square Delta conjecture \cite{DAdderio-Iraci-VandenWyngaerd-DeltaSquare-2019} is still open. 

On the valley Delta conjecture itself, almost nothing is known: for example it is not even clear that the combinatorial side (the rightmost sum in \eqref{eq:DeltaConjs}) is a symmetric function.


In this work we make a first step in this direction, by proving the so called Schr\"oder case of the valley versions of the Delta conjecture, i.e.\ the scalar product $\<-,e_{n-d}h_d\>$. The strategy, similar to the one used in \cite{DAdderio-Iraci-VandenWyngaerd-TheBible-2019}, allows also to prove the Schr\"oder case of the valley Delta square conjecture in \cite{Iraci-VandenWyngaerd-Valley-Square-2021}. Using results from \cite{Iraci-VandenWyngaerd-pushing-2021} we are able also to deduce the Catalan case (i.e.\ the scalar product  $\<-,e_{n}\>$) of the extended valley Delta conjecture. Finally, assuming the symmetry of (a refinement of) the combinatorial side of the extended valley Delta conjecture, we deduce also the Catalan case of its square version \cite{Iraci-VandenWyngaerd-Valley-Square-2021}.

\smallskip

The paper is organized in the following way: in Section~2 and Section~3 we introduce the combinatorial objects and the symmetric function tools respectively, needed in the rest of the paper. In Section~4 we prove our main results, providing in particular a recursion leading to the proof of the Schr\"oder cases of the valley Delta and the valley Delta square conjectures. In Sections~5 we prove how the Schr\"oder case of the valley Delta conjecture implies the Catalan case of the extended valley Delta conjecture, and in Section~6, assuming the symmetry of (a refinement of) the combinatorial side of the extended valley Delta conjecture, we deduce also the Catalan case of its square version.

\section{Combinatorial definitions}\label{sec:CombDef}

We recall the relevant combinatorial definitions (see also \cites{Haglund-Remmel-Wilson-2018, Qiu-Wilson-2020, Iraci-VandenWyngaerd-Valley-Square-2021}).

\begin{definition}
	A \emph{square path} of size $n$ is a lattice path going from $(0,0)$ to $(n,n)$ consisting of east or north unit steps, always ending with an east step. The set of such paths is denoted by $\SQ(n)$. The \emph{shift} of a square path is the maximum value $s$ such that  the path intersect the line $y=x-s$ in at least one point. We refer to the line $y=x+i$ as \emph{$i$-th diagonal} of the path, to the line $x=y$ (the $0$-th diagonal) as the \emph{main diagonal}, and to the line $y=x-s$ where $s$ is the shift of the square path as the \emph{base diagonal}. A vertical step whose starting point lies on the $i$-th diagonal is said to be at \emph{height} $i$. A \emph{Dyck path} is a square path whose shift is $0$. The set of Dyck paths is denoted by $\D(n)$. Of course $\D(n) \subseteq \SQ(n)$.  
\end{definition}

For example, the path (ignoring the circled numbers) in Figure~\ref{fig:labelled-square-path} has shift $3$. 

\begin{figure}[!ht]
	\centering
	\begin{tikzpicture}[scale = 0.6]
		\draw[step=1.0, gray!60, thin] (0,0) grid (8,8);
		\draw[gray!60, thin] (3,0) -- (8,5);
		
		\draw[blue!60, line width=1.6pt] (0,0) -- (0,1) -- (1,1) -- (2,1) -- (3,1) -- (4,1) -- (4,2) -- (5,2) -- (5,3) -- (5,4) -- (6,4) -- (6,5) -- (6,6) -- (6,7) -- (7,7) -- (7,8) -- (8,8);
		
		\node at (0.5,0.5) {$2$};
		\draw (0.5,0.5) circle (.4cm); 
		\node at (4.5,1.5) {$1$};
		\draw (4.5,1.5) circle (.4cm); 
		\node at (5.5,2.5) {$2$};
		\draw (5.5,2.5) circle (.4cm); 
		\node at (5.5,3.5) {$4$};
		\draw (5.5,3.5) circle (.4cm); 
		\node at (6.5,4.5) {$1$};
		\draw (6.5,4.5) circle (.4cm); 
		\node at (6.5,5.5) {$3$};
		\draw (6.5,5.5) circle (.4cm); 
		\node at (6.5,6.5) {$4$};
		\draw (6.5,6.5) circle (.4cm); 
		\node at (7.5,7.5) {$1$};
		\draw (7.5,7.5) circle (.4cm);
	\end{tikzpicture}
	\caption{Example of an element in $\LSQ(0,8)$.}
	\label{fig:labelled-square-path}
\end{figure}
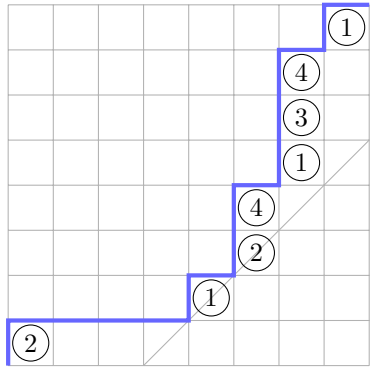

\begin{definition}
	Let $\pi$ be a square path of size $n$. We define its \emph{area word} to be the sequence of integers $a(\pi) = (a_1(\pi), a_2(\pi), \cdots, a_n(\pi))$ such that the $i$-th vertical step of the path starts from the diagonal $y=x+a_i(\pi)$. For example the path in Figure~\ref{fig:labelled-square-path} has area word $(0, \, -\!3, \, -\!3, \, -\!2, \, -\!2, \, -\!1, \, 0, \, 0)$.
\end{definition}

\begin{definition}
	A \emph{partial labelling} of a square path $\pi$ of size $n$ is an element $w \in \mathbb N^n$ such that
	\begin{itemize}
		\item if $a_i(\pi) > a_{i-1}(\pi)$, then $w_i > w_{i-1}$,
		\item $a_1(\pi) = 0 \implies w_1 > 0$,
		\item there exists an index $i$ such that $a_i(\pi) = - \shift(\pi)$ and $w_i(\pi) > 0$,
	\end{itemize}
	i.e.\ if we label the $i$-th vertical step of $\pi$ with $w_i$, then the labels appearing in each column of $\pi$ are strictly increasing from bottom to top, with the additional restrictions that, if the path starts north then the first label cannot be a $0$, and that there is at least one positive label lying on the base diagonal.
	
	We omit the word \emph{partial} if the labelling consists of strictly positive labels only.
\end{definition}

\begin{definition}
	A \emph{(partially) labelled square path} (resp.\ \emph{Dyck path}) is a pair $(\pi, w)$ where $\pi$ is a square path (resp. Dyck path) and $w$ is a (partial) labelling of $\pi$. We denote by $\LSQ(m,n)$ (resp. $\LD(m,n)$) the set of labelled square paths (resp. Dyck paths) of size $m+n$ with exactly $n$ positive labels, and thus exactly $m$ labels equal to $0$. See Figure~\ref{fig:labelled-square-path} for an example.
\end{definition}

The following definitions will be useful later on.

\begin{definition}
	Let $w$ be a labelling of square path of size $n$. We define 
	\[  x^w \coloneqq \left. \prod_{i=1}^{n} x_{w_i} \right\rvert_{x_0 = 1}.\]
\end{definition}

The fact that we set $x_0 = 1$ explains the use of the expression \emph{partially labelled}, as the labels equal to $0$ do not contribute to the monomial.

Sometimes, with an abuse of notation, we will write $\pi$ as a shorthand for a labelled path $(\pi, w)$. In that case, we use the notation 
\[x^\pi \coloneqq x^w.\]

Now we want to extend our sets introducing some decorations.

\begin{definition}
	\label{def:valley}
	The \emph{contractible valleys} of a labelled square path $\pi$ of size $n$ are the indices $1 \leq i \leq n$ such that one of the following holds:
	\begin{itemize}
		\item $i = 1$ and either $a_1(\pi) < -1$, or $a_1(\pi) = -1$ and $w_1 > 0$,
		\item $i > 1$ and $a_i(\pi) < a_{i-1}(\pi)$,
		\item $i > 1$, $a_i(\pi) = a_{i-1}(\pi)$, and $w_i > w_{i-1}$.
	\end{itemize}
	
	We define \[ v(\pi, w) \coloneqq \{1 \leq i \leq n \mid i \text{ is a contractible valley} \}, \] corresponding to the set of vertical steps that are directly preceded by a horizontal step and, if we were to remove that horizontal step and move it after the vertical step, we would still get a square path with a valid labelling. In particular, if the vertical step is in the first row and it is attached to a $0$ label, then we require that it is preceded by at least two horizontal steps (as otherwise by removing it we get a path starting north with a $0$ label in the first row).
\end{definition}


\begin{definition}
	\label{def:rise}
	The \emph{rises} of a (labelled) square path $\pi$ of size $n$ are the indices in \[ r(\pi) \coloneqq \{2 \leq i \leq n \mid a_i(\pi) > a_{i-1}(\pi)\}, \] i.e.\ the vertical steps that are directly preceded by another vertical step. 
\end{definition}

\begin{definition}
	A \emph{valley-decorated (partially) labelled square path} is a triple $(\pi, w, dv)$ where $(\pi, w)$ is a (partially) labelled square path and $dv \subseteq v(\pi, w)$. A \emph{rise-decorated (partially) labelled square path} is a triple $(\pi, w, dr)$ where $(\pi, w)$ is a (partially) labelled square path and $dr \subseteq r(\pi)$.
\end{definition}

Again, we will often write $\pi$ as a shorthand for the corresponding triple $(\pi, w, dv)$ or $(\pi, w, dr)$.
%

We denote by $\LSQ(m,n)^{\bullet k}$ (resp. $\LSQ(m,n)^{\ast k}$) the set of partially labelled valley-decorated (resp. rise-decorated) square paths of size $m+n$ with $n$ positive labels and $k$ decorated contractible valleys (resp. decorated rises). We denote by $\LD(m,n)^{\bullet k}$ (resp. $\LD(m,n)^{\ast k}$) the corresponding subset of Dyck paths. 

We also define $\LSQ'(m,n)^{\bullet k}$ as the set of paths in $\LSQ(m,n)^{\bullet k}$ such that there exists an index $i$ such that $a_i(\pi) = - \shift(\pi)$ with $i \not \in dv$ and $w_i(\pi) > 0$, i.e.\ there is at least one positive label lying on the base diagonal that is not a decorated valley. See Figure~\ref{fig:decorated-square-paths} for examples.

Notice that, because of the restrictions we have on the labelling and the decorations, the only path with $n=0$ is the empty path, for which $m=0$ and $k=0$.

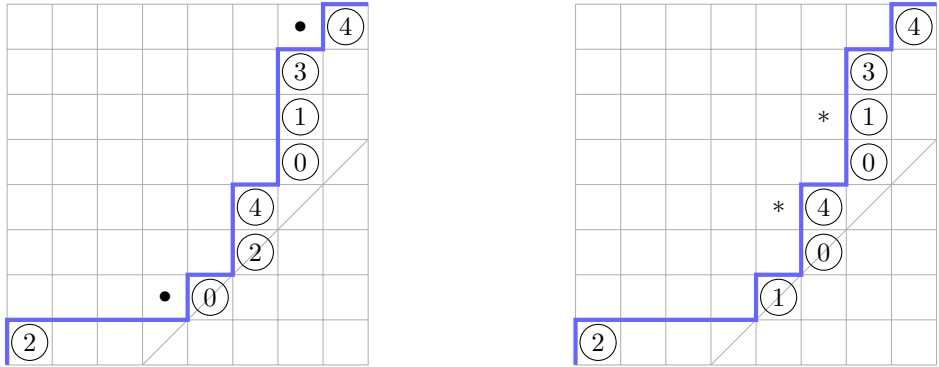
\begin{figure}[!ht]
    \centering
	\begin{minipage}{0.5\textwidth}
		\centering
		\begin{tikzpicture}[scale = 0.6]
		\draw[step=1.0, gray!60, thin] (0,0) grid (8,8);
		\draw[gray!60, thin] (3,0) -- (8,5);
		
		\draw[blue!60, line width=1.6pt] (0,0) -- (0,1) -- (1,1) -- (2,1) -- (3,1) -- (4,1) -- (4,2) -- (5,2) -- (5,3) -- (5,4) -- (6,4) -- (6,5) -- (6,6) -- (6,7) -- (7,7) -- (7,8) -- (8,8);
		
		\node at (3.5,1.5) {$\bullet$};
		\node at (6.5,7.5) {$\bullet$};
		
		\node at (0.5,0.5) {$2$};
		\draw (0.5,0.5) circle (.4cm); 
		\node at (4.5,1.5) {$0$};
		\draw (4.5,1.5) circle (.4cm); 
		\node at (5.5,2.5) {$2$};
		\draw (5.5,2.5) circle (.4cm); 
		\node at (5.5,3.5) {$4$};
		\draw (5.5,3.5) circle (.4cm); 
		\node at (6.5,4.5) {$0$};
		\draw (6.5,4.5) circle (.4cm); 
		\node at (6.5,5.5) {$1$};
		\draw (6.5,5.5) circle (.4cm); 
		\node at (6.5,6.5) {$3$};
		\draw (6.5,6.5) circle (.4cm); 
		\node at (7.5,7.5) {$4$};
		\draw (7.5,7.5) circle (.4cm);
		\end{tikzpicture}
	\end{minipage}%
	\begin{minipage}{0.5\textwidth}
		\centering
		\begin{tikzpicture}[scale = 0.6]
		\draw[step=1.0, gray!60, thin] (0,0) grid (8,8);
		\draw[gray!60, thin] (3,0) -- (8,5);
		
		\draw[blue!60, line width=1.6pt] (0,0) -- (0,1) -- (1,1) -- (2,1) -- (3,1) -- (4,1) -- (4,2) -- (5,2) -- (5,3) -- (5,4) -- (6,4) -- (6,5) -- (6,6) -- (6,7) -- (7,7) -- (7,8) -- (8,8);
		
		\node at (4.5,3.5) {$\ast$};			
		\node at (5.5,5.5) {$\ast$};
		
		\node at (0.5,0.5) {$2$};
		\draw (0.5,0.5) circle (.4cm); 
		\node at (4.5,1.5) {$1$};
		\draw (4.5,1.5) circle (.4cm); 
		\node at (5.5,2.5) {$0$};
		\draw (5.5,2.5) circle (.4cm); 
		\node at (5.5,3.5) {$4$};
		\draw (5.5,3.5) circle (.4cm); 
		\node at (6.5,4.5) {$0$};
		\draw (6.5,4.5) circle (.4cm); 
		\node at (6.5,5.5) {$1$};
		\draw (6.5,5.5) circle (.4cm); 
		\node at (6.5,6.5) {$3$};
		\draw (6.5,6.5) circle (.4cm); 
		\node at (7.5,7.5) {$4$};
		\draw (7.5,7.5) circle (.4cm);
		\end{tikzpicture}
	\end{minipage}
	\caption{Examples of an element in $\LSQ'(2,6)^{\bullet 2}$ (left) and an element in $\LSQ(2,6)^{\ast 2}$ (right).}
	\label{fig:decorated-square-paths}
\end{figure}

We also recall the two relevant statistics on these sets (see \cite{Iraci-VandenWyngaerd-Valley-Square-2021}) that reduce to the ones defined in \cite{Loehr-Warrington-square-2007} when $m=0$ and $k=0$.

\begin{definition}
	\label{def:area}
	Let $(\pi, w, dr) \in \LSQ(m,n)^{\ast k}$ and $s$ be its shift. We define 
	\[ \area(\pi, w, dr) \coloneqq \sum_{i \not \in dr} (a_i(\pi) + s), \]
	i.e.\ the number of whole squares between the path and the base diagonal that are not in rows containing a decorated rise.
	
	For $(\pi, w, dv) \in \LSQ(m,n)^{\bullet k}$ we set $\area(\pi, w, dv) \coloneqq \area(\pi, w, \varnothing)$, where $(\pi, w, \varnothing) \in \LSQ(m,n)^{\ast 0}$.
\end{definition}

For example, the paths in Figure~\ref{fig:decorated-square-paths} have area $13$ (left) and $10$ (right). Notice that the area does not depend on the labelling.

\begin{definition}
	\label{def:dinv}
	Let $(\pi, w, dv) \in \LSQ(m,n)^{\bullet k}$. For $1 \leq i < j \leq n$, the pair $(i,j)$ is a \emph{diagonal inversion} if
	\begin{itemize}
		\item either $a_i(\pi) = a_j(\pi)$ and $w_i < w_j$ (\emph{primary inversion}),
		\item or $a_i(\pi) = a_j(\pi) + 1$ and $w_i > w_j$ (\emph{secondary inversion}),
	\end{itemize}
	where $w_i$ denotes the $i$-th letter of $w$, i.e.\ the label of the vertical step in the $i$-th row. Then we define 
	\begin{align*}
	\dinv(\pi,w,dv) & \coloneqq \# \{ 1 \leq i < j \leq n \mid (i,j) \text{ is an inversion  and } i \not \in dv \}\\
	& + \#\{1 \leq i \leq n \mid a_i(\pi) < 0 \text{ and } w_i > 0 \} - \# dv.
	\end{align*}
	
	For $(\pi, w, dr) \in \LSQ(m,n)^{\ast k}$ we set $\dinv(\pi, w, dr) \coloneqq \dinv(\pi, w, \varnothing)$, where $(\pi, w, \varnothing) \in \LSQ(m,n)^{\bullet 0}$.
\end{definition}

We refer to the middle term, counting the non-zero labels below the main diagonal, as \emph{bonus} or \emph{tertiary dinv}.

For example, the path in Figure~\ref{fig:decorated-square-paths} (left) has dinv equal to $4$: $2$ primary inversions in which the leftmost label is not a decorated valley, i.e.\ $(1,7)$ and $(1,8)$; $1$ secondary inversion in which the leftmost label is not a decorated valley, i.e.\ $(1,6)$; $3$ bonus dinv, coming from the rows $3$, $4$, and $6$; $2$ decorated contractible valleys.

It is easy to check that if $j \in dv$ then either there exists some diagonal inversion $(i,j)$ or $a_j(\pi) < 0$, and so the dinv is always non-negative (see \cite{Iraci-VandenWyngaerd-Valley-Square-2021}*{Proposition~1}).

\smallskip

\emph{To lighten the notation we will usually refer to a labelled path $(\pi, w, dr)$ or $(\pi, w, dv)$ simply by $\pi$, so for example we will write $\pi \in \LSQ(m,n)^{\ast k}$ and $\dinv(\pi)$.}

\smallskip

Let $\pi$ be any labelled path defined above, with shift $s$. We define its \emph{reading word} as the sequence of labels read starting from the ones in the base diagonal ($y=x-s$) going bottom to top; next the ones in the diagonal $y=x-s+1$ bottom to top; then the ones in the diagonal $y=x-s+2$ and so on. For example, the path in Figure~\ref{fig:decorated-square-paths} (left) has reading word $02401234$.

Let us consider the paths in $\LSQ(m,n)^{\bullet k}$ where the reading word is a shuffle of $m$ $0$'s, the string $1, 2,\cdots, n-d$, and the string $n,n-1, \cdots n-d+1$. Notice that, given this restriction and the information about the position of the zero labels, all the information we need to keep track of the labelling is the position of the $d$ biggest labels, which will end up labelling \emph{peaks}, i.e.\ vertical steps followed by a horizontal step. Hence we can denote the $d$ biggest labels by decorated peaks and forget about the positive labels. We can thus identify our set with the set $\SQ(m,n)^{\bullet k, \circ d}$ of square paths with $m$ vertical steps labelled by a zero, $n$ non-labelled vertical steps, $k$ decorated contractible valleys and $d$ decorated peaks, where a valley is contractible if it is contractible in the corresponding labelled path, and the statistics are inherited from the labelled path as well. Similarly, we define $\SQ'(m,n)^{\bullet k, \circ d} \subseteq \SQ(m,n)^{\bullet k, \circ d}$ to be the subset coming from $\LSQ'(m,n)^{\bullet k}\subseteq \LSQ(m,n)^{\bullet k}$, $\SQ(m,n)^{\ast k, \circ d}$ the one coming from rise-decorated paths, and $\D(m,n)^{\bullet k, \circ d}$ and $\D(m,n)^{\ast k, \circ d}$ for the Dyck counterparts. 

Finally, we sometimes omit writing $m$ or $k$ when they are equal to $0$, e.g.
\[\SQ'(n)^{\bullet k, \circ d}=\SQ'(0,n)^{\bullet k, \circ d}\quad\text{ and }\quad \LSQ(n) =\LSQ(0,n)^{\ast 0}. \] 

\section{Symmetric functions}\label{sec:SF}

For all the undefined notations and the unproven identities, we refer to \cite{DAdderio-Iraci-VandenWyngaerd-TheBible-2019}*{Section~1}, where definitions, proofs and/or references can be found. 

We denote by $\Lambda$ the graded algebra of symmetric functions with coefficients in $\mathbb{Q}(q,t)$, and by $\<\, , \>$ the \emph{Hall scalar product} on $\Lambda$, defined by declaring that the Schur functions form an orthonormal basis.

The standard bases of the symmetric functions that will appear in our calculations are the monomial $\{m_\lambda\}_{\lambda}$, complete $\{h_{\lambda}\}_{\lambda}$, elementary $\{e_{\lambda}\}_{\lambda}$, power $\{p_{\lambda}\}_{\lambda}$ and Schur $\{s_{\lambda}\}_{\lambda}$ bases.

For a partition $\mu \vdash n$, we denote by \[ \Ht_\mu \coloneqq \Ht_\mu[X] = \Ht_\mu[X; q,t] = \sum_{\lambda \vdash n} \widetilde{K}_{\lambda, \mu}(q,t) s_{\lambda} \] the \emph{(modified) Macdonald polynomials}, where \[ \widetilde{K}_{\lambda, \mu}(q,t) \coloneqq K_{\lambda, \mu}(q,1/t) t^{n(\mu)} \] are the \emph{(modified) $q,t$-Kostka coefficients} (see \cite{Haglund-Book-2008}*{Chapter~2} for more details). 

Macdonald polynomials form a basis of the ring of symmetric functions $\Lambda$. This is a modification of the basis introduced by Macdonald \cite{Macdonald-Book-1995}.

If we identify the partition $\mu$ with its Ferrer diagram, i.e.\ with the collection of cells $\{(i,j)\mid 1\leq i\leq \mu_j, 1\leq j\leq \ell(\mu)\}$, then for each cell $c\in \mu$ we refer to the \emph{arm}, \emph{leg}, \emph{co-arm} and \emph{co-leg} (denoted respectively as $a_\mu(c), l_\mu(c), a'_\mu(c), l'_\mu(c)$) as the number of cells in $\mu$ that are strictly to the right, above, to the left and below $c$ in $\mu$, respectively (see Figure~\ref{fig:limbs}). 

\begin{figure}
	\centering
	\begin{tikzpicture}[scale=.52]
		\draw[gray] (0,0) grid (9,1);
		\draw[gray] (0,1) grid (8,2);
		\draw[gray] (0,2) grid (7,3);
		\draw[gray] (0,3) grid (7,4);
		\draw[gray] (0,4) grid (4,5);
		\draw[gray] (0,5) grid (4,6);
		\draw[gray] (0,6) grid (3,7);
		\draw[gray] (0,7) grid (2,8);
		\node at (2.5,3.5) {$c$};
		\fill[blue, opacity=.4] (0,3) rectangle (2,4) node[midway, opacity=1, black]{co-arm};
		\fill[blue, opacity=.4] (3,3) rectangle (7,4) node[midway, opacity=1, black]{arm};
		\fill[blue, opacity=.4] (2,4) rectangle (3,7) node[midway, opacity=1, black, rotate=90]{leg};
		\fill[blue, opacity=.4] (2,3) rectangle (3,0) node[midway, opacity=1, black, rotate=90]{co-leg};
	\end{tikzpicture}
	\caption{Limbs and co-limbs of a cell in a partition.}
	\label{fig:limbs}
\end{figure}
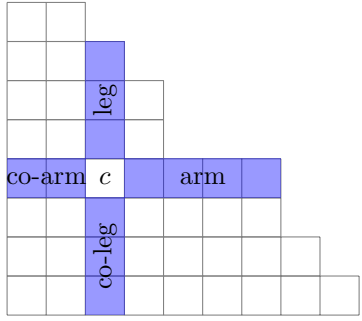

Let $M \coloneqq (1-q)(1-t)$. For every partition $\mu$, we define the following constants:

\begin{align*}
	B_{\mu} & \coloneqq B_{\mu}(q,t) = \sum_{c \in \mu} q^{a_{\mu}'(c)} t^{l_{\mu}'(c)},\\
	\Pi_{\mu} &  \coloneqq \Pi_{\mu}(q,t) = \prod_{c \in \mu / (1)} (1-q^{a_{\mu}'(c)} t^{l_{\mu}'(c)}),\\
	w_{\mu} &\coloneqq w_{\mu}(q,t)=\prod_{c\in \mu}(q^{a_{\mu}(c)}-t^{l_{\mu}(c)+1})(t^{l_{\mu}(c)}-q^{a_{\mu}(c)+1}). 
	\end{align*}

We will make extensive use of the \emph{plethystic notation} (cf. \cite{Haglund-Book-2008}*{Chapter~1 page 19}). We will use also the standard shorthand $f^\ast = f \left[\frac{X}{M}\right]$.

We define the \emph{star scalar product} by setting for every $f,g\in \Lambda$
\[ \<f,g\>_\ast\coloneqq \<f[X],\omega g[MX]\>, \]
where $\omega$ is the involution of $\Lambda$ sending $e_\lambda$ to $h_\lambda$ for every partition $\lambda$.

It is well known that for any two partitions $\mu,\nu$ we have 
\[ \<\Ht_\mu,\Ht_\nu\>_\ast=\delta_{\mu,\nu}w_\mu. \]

We also need several linear operators on $\Lambda$.

\begin{definition}[\protect{\cite{Bergeron-Garsia-ScienceFiction-1999}*{3.11}}]
	\label{def:nabla}
	We define the linear operator $\nabla \colon \Lambda \rightarrow \Lambda$ on the eigenbasis of Macdonald polynomials as \[ \nabla \Ht_\mu = e_{\lvert \mu \rvert}[B_\mu] \Ht_\mu=q^{n(\mu')}t^{n(\mu)} \Ht_\mu. \]
\end{definition}

\begin{definition}
	\label{def:pi}
	We define the linear operator $\mathbf{\Pi} \colon \Lambda \rightarrow \Lambda$ on the eigenbasis of Macdonald polynomials as \[ \mathbf{\Pi} \Ht_\mu = \Pi_\mu \Ht_\mu \] where we conventionally set $\Pi_{\varnothing} \coloneqq 1$.
\end{definition}

\begin{definition}
	\label{def:delta}
	For $f \in \Lambda$, we define the linear operators $\Delta_f, \Delta'_f \colon \Lambda \rightarrow \Lambda$ on the eigenbasis of Macdonald polynomials as \[ \Delta_f \Ht_\mu = f[B_\mu] \Ht_\mu, \qquad \qquad \Delta'_f \Ht_\mu = f[B_\mu-1] \Ht_\mu. \]
\end{definition}

Observe that on the vector space of symmetric functions homogeneous of degree $n$, denoted by $\Lambda^{(n)}$, the operator $\nabla$ equals $\Delta_{e_n}$. Notice also that $\nabla$, $\Delta_f$ and $\mathbf{\Pi}$ are all self-adjoint with respect to the star scalar product.

\begin{definition}[\protect{\cite{DAdderio-Iraci-VandenWyngaerd-Theta-2021}*{(28)}}]
	\label{def:theta}
	 For any symmetric function $f \in \Lambda^{(n)}$ we define the \emph{Theta operators} on $\Lambda$ in the following way: for every $F \in \Lambda^{(m)}$ we set
	\begin{equation*}
		\Theta_f F  \coloneqq 
		\left\{\begin{array}{ll}
			0 & \text{if } n \geq 1 \text{ and } m=0 \\
			f \cdot F & \text{if } n=0 \text{ and } m=0 \\
			\mathbf{\Pi} (f \left[\frac{X}{M}\right] \cdot \mathbf{\Pi}^{-1} F) & \text{otherwise}
		\end{array}
		\right. ,
	\end{equation*}
and we extend by linearity the definition to any $f, F \in \Lambda$.
\end{definition}

It is clear that $\Theta_f$ is linear, and moreover, if $f$ is homogeneous of degree $k$, then so is $\Theta_f$, i.e. \[\Theta_f \Lambda^{(n)} \subseteq \Lambda^{(n+k)} \qquad \text{ for } f \in \Lambda^{(k)}. \]

It is convenient to introduce the so called $q$-notation. In general, a $q$-analogue of an expression is a generalisation involving a parameter $q$ that reduces to the original one for $q \rightarrow 1$.

\begin{definition}
	For a natural number $n \in \mathbb{N}$, we define its $q$-analogue as \[ [n]_q \coloneqq \frac{1-q^n}{1-q} = 1 + q + q^2 + \dots + q^{n-1}. \]
\end{definition}

Given this definition, one can define the $q$-factorial and the $q$-binomial as follows.

\begin{definition}
	We define \[ [n]_q! \coloneqq \prod_{k=1}^{n} [k]_q \quad \text{and} \quad \qbinom{n}{k}_q \coloneqq \frac{[n]_q!}{[k]_q![n-k]_q!} \]
\end{definition}

\begin{definition}
	For $x$ any variable and $n \in \N \cup \{ \infty \}$, we define the \emph{$q$-Pochhammer symbol} as \[ (x;q)_n \coloneqq \prod_{k=0}^{n-1} (1-xq^k) = (1-x) (1-xq) (1-xq^2) \cdots (1-xq^{n-1}). \]
\end{definition}

We can now introduce yet another family of symmetric functions.

\begin{definition}
	\label{def:Enk}
	For $0 \leq k \leq n$, we define the symmetric function $E_{n,k}$ \cite{Garsia-Haglund-qtCatalan-2002} by the expansion \[ e_n \left[ X \frac{1-z}{1-q} \right] = \sum_{k=0}^n \frac{(z;q)_k}{(q;q)_k} E_{n,k}. \]
\end{definition}

Notice that $E_{n,0} = \delta_{n,0}$. Setting $z=q^j$ we get \[ e_n \left[ X \frac{1-q^j}{1-q} \right] = \sum_{k=0}^n \frac{(q^j;q)_k}{(q;q)_k} E_{n,k} = \sum_{k=0}^n \qbinom{k+j-1}{k}_q E_{n,k} \] and in particular, for $j=1$, we get 
\begin{equation}
	\label{eq:Enken} e_n = E_{n,0} + E_{n,1} + E_{n,2} + \cdots + E_{n,n}, 
\end{equation} 
so these symmetric functions split $e_n$, in some sense.

The Theta operators will be useful to restate the Delta conjectures in a new fashion, thanks to the following results.

\begin{theorem}[\protect{\cite{DAdderio-Iraci-VandenWyngaerd-Theta-2021}*{Theorem~3.1} }]
	\label{thm:theta-en}
	\[ \Theta_{e_k} \nabla e_{n-k} = \Delta'_{e_{n-k-1}} e_n \]
\end{theorem}

We will also need the following identity.

\begin{theorem}[\cite{DAdderio-Romero-Theta-Identities-2020}*{Corollary~9.2}]
	\label{thm:sf-identity}
	Given $m,n,k,r \in \N$, we have
	\begin{align*}
		h_m^\perp \Theta_{e_k} \nabla E_{n-k,r} = \sum_{p=0}^m t^{m-p} \sum_{i=0}^p q^{\binom{i}{2}} \qbinom{r-p+i}{i}_q \qbinom{r}{p-i}_q \Delta_{h_{m-p}} \Theta_{e_{k-i}} \nabla E_{n-m-(k-i), r-p+i}
	\end{align*}
where $h_m^\perp$ is the adjoint of the multiplication by $h_m$ with respect to the Hall scalar product.
\end{theorem}

We applied the change of variables $j \mapsto m, m \mapsto k, p \mapsto n-k-r, k \mapsto r, s \mapsto p, r \mapsto p-i$ in \cite{DAdderio-Romero-Theta-Identities-2020}*{Corollary~9.2} in order to make it easier to interpret combinatorially and more consistent with the notation used in \cite{Iraci-VandenWyngaerd-pushing-2021}.

\section{The Schr\"oder case of the valley version of the Delta conjectures}

First of all, we state the extended valley Delta conjecture and the extended valley Delta square conjecture.
\begin{conjecture}[Extended valley Delta conjecture \cite{Qiu-Wilson-2020}]
	\label{conj:valleyDelta}
	\begin{equation} \label{eq:valleyDelta}
		\Delta_{h_m}\Theta_{e_k}\nabla e_{n-k} =\sum_{\pi\in \LD(m,n)^{\bullet k}}q^{\dinv(\pi)}t^{\area(\pi)}x^{\pi}.
	\end{equation}
\end{conjecture}

\begin{conjecture}[Extended valley Delta square conjecture \cite{Iraci-VandenWyngaerd-Valley-Square-2021}]
	\begin{equation} \label{eq:valleyDeltasquare}
		\Delta_{h_m}\Theta_{e_k}\nabla \omega (p_{n-k})=\sum_{\pi\in \LSQ'(m,n)^{\bullet k}}q^{\dinv(\pi)}t^{\area(\pi)}x^{\pi}.	
	\end{equation}
\end{conjecture}

\begin{remark}
	\label{rmk:SFcombside}
	It should be noticed that in general the combinatorial sides of \eqref{eq:valleyDelta} and \eqref{eq:valleyDeltasquare} are not even known to be symmetric functions. Hence these conjectures include the statement that those combinatorial sums are indeed symmetric functions.
\end{remark}

The main result we want to prove is the so-called \emph{Schr\"oder case} of the valley version of the Delta conjecture and the Delta square conjecture. In other words we want to show that the identities hold if we take the scalar product with $e_{n-d}h_d$. On the combinatorial side, assuming that those sums are symmetric functions (cf.\ Remark~\ref{rmk:SFcombside}), the theory of \emph{shuffles} (cf.\ \cite{DAdderio-Iraci-VandenWyngaerd-TheBible-2019}*{Section~3.3}) tells us that taking the scalar product with $e_{n-d}h_d$ in \eqref{eq:valleyDelta} and \eqref{eq:valleyDeltasquare} gives (cf.\ end of Section~\ref{sec:CombDef})
\[\sum_{\pi \in \D(m,n)^{\bullet k, \circ d}} q^{\dinv(\pi)} t^{\area(\pi)} \quad \text{ and } \quad \sum_{\pi \in \SQ'(m,n)^{\bullet k, \circ d}} q^{\dinv(\pi)} t^{\area(\pi)}\]
respectively.
\begin{theorem}
	\label{thm:valley-schroeder}
	\[ \< \Theta_{e_k} \nabla e_{n-k}, e_{n-d} h_d \> = \sum_{\pi \in \D(n)^{\bullet k, \circ d}} q^{\dinv(\pi)} t^{\area(\pi)} \]
\end{theorem}

\begin{theorem}
	\label{thm:valley-square-schroeder}
	\[ \< \Theta_{e_k} \nabla \omega(p_{n-k}), e_{n-d} h_d \> = \sum_{\pi \in \SQ'(n)^{\bullet k, \circ d}} q^{\dinv(\pi)} t^{\area(\pi)} \]
\end{theorem}

In order to prove these results, we proceed as follows. First, we recall (\eqref{eq:Enken} and \cite{DAdderio-Iraci-VandenWyngaerd-Theta-2021}*{(26)}) that \[ e_{n-k} = \sum_{r=1}^{n-k} E_{n-k,r} \qquad \text{ and } \qquad \omega(p_{n-k}) = \sum_{r=1}^{n-k} \frac{[n-k]_q}{[r]_q} E_{n-k,r} \]
when $n-k > 0$, and $e_0 = \omega(p_0) = 1$ when $n-k=0$.

Then, we find an algebraic recursion satisfied by the polynomials $\< \Theta_{e_k} \nabla E_{n-k,r}, e_{n-d} h_d \>$, which we use to derive a similar one satisfied by the polynomials $\< \Theta_{e_k} \nabla \frac{[n-k]_q}{[r]_q} E_{n-k,r}, e_{n-d} h_d \>$.

Next, we prove that the $q,t$-enumerators of the sets \[ \D(n \backslash r)^{\bullet k, \circ d} \coloneqq \{ \pi \in \D(n)^{\bullet k, \circ d} \mid \text{$r$ non-decorated vertical steps of $\pi$ touch the main diagonal} \} \] and \[ \SQ'(n \backslash r)^{\bullet k, \circ d} \coloneqq \{ \pi \in \SQ'(n)^{\bullet k, \circ d} \mid \text{$r$ non-decorated vertical steps of $\pi$ touch the base diagonal} \} \] satisfy the same recursions as the aforementioned polynomials with the same initial conditions, thus proving the equality.

Finally, we take the sum over $r$ of the identities we get, completing the proof of Theorem~\ref{thm:valley-schroeder} and Theorem~\ref{thm:valley-square-schroeder}.

\subsection{Algebraic recursions}

We begin by proving the following lemma.

\begin{lemma}
	\[ \< \Theta_{e_k} \nabla E_{n-k, r}, e_{n-d} h_d \> = \< \Delta_{h_k} \Delta_{e_{n-k-d}} E_{n-k, r}, e_{n-k} \> \]
\end{lemma}

\begin{proof}
	By \cite{DAdderio-Iraci-VandenWyngaerd-Theta-2021}*{Lemma~6.1}, we have \[ \< \Theta_{e_k} \nabla E_{n-k, r}, e_{n-d} h_d \> = \< \Delta_{h_k} \nabla E_{n-k, r}, e_{n-k-d} h_d \> \] and by repeatedly applying the well-known identity $\< \Delta_{e_a} f, h_n \> = \< f, e_a h_{n-a} \>$ (see \cite{DAdderio-Iraci-VandenWyngaerd-TheBible-2019}*{Lemma~4.1}), we have
	\begin{align*}
		\< \Delta_{h_k} \nabla E_{n-k, r}, e_{n-k-d} h_d \> & = \< \Delta_{e_{n-k-d}} \Delta_{h_k} \nabla E_{n-k, r}, h_{n-k} \> \\
		& = \< \nabla \Delta_{h_k} \Delta_{e_{n-k-d}} E_{n-k, r}, h_{n-k} \> \\
		& = \< \Delta_{h_k} \Delta_{e_{n-k-d}} E_{n-k, r}, e_{n-k} \> 
	\end{align*}
	as desired.
\end{proof}

Now, the polynomial $\< \Delta_{h_k} \Delta_{e_{n-k-d}} E_{n-k, r}, e_{n-k} \>$ coincides with the expression $F_{n,r}^{(d,k)}$ in \cite{DAdderio-Iraci-VandenWyngaerd-TheBible-2019}*{Section~4.3, (4.77)}, so by \cite{DAdderio-Iraci-VandenWyngaerd-TheBible-2019}*{Theorem~4.18}, up to some simple rewriting, we have the following.

\begin{proposition}
	\label{prop:preliminary-alg-recursion}
	The expressions $\< \Theta_{e_k} \nabla E_{n-k, r}, e_{n-d} h_d \>$ satisfy the recursion
	\begin{align*}
		\< \Theta_{e_k} \nabla E_{n-k, r}, e_{n-d} h_d \> & = \sum_{j=0}^k t^{n-r-j} \sum_{s=0}^{n-r} \sum_{v=0}^d q^{\binom{v}{2}} \qbinom{r}{v}_q \qbinom{r+j-1}{j} \qbinom{v+j+s-1}{s} \\
		& \quad \times \< \Theta_{e_{k-j}} \nabla E_{n-k-r, s}, e_{(n-r-j)-(d-v)} h_{d-v} \> 
	\end{align*}
	with initial conditions $\< \Theta_{e_k} \nabla E_{n-k, r}, e_{n-d} h_d \> = \delta_{r,0} \delta_{k,0} \delta_{d,0}$.
\end{proposition}

Notice that this implies that $\< \Theta_{e_k} \nabla E_{n-k, r}, e_{n-d} h_d \>$ is actually a polynomial in $\N[q,t]$. We want to rewrite it slightly, for which we need the following lemma.

\begin{lemma}
	\label{lem:chu-vandermonde}
	\[ q^{\binom{v}{2}} \qbinom{r}{v}_q \qbinom{r+j-1}{j} = \sum_{u=0}^{r-v} q^{\binom{u}{2}} \qbinom{u+v}{u}_q q^{\binom{u+v}{2}} \qbinom{r}{u+v}_q \qbinom{v+j-1}{j-u}_q \]
\end{lemma}

\begin{proof}
	We have
	\begin{align*}
		q^{\binom{v}{2}} \qbinom{r}{v}_q \qbinom{r+j-1}{j} & = q^{\binom{v}{2}} \qbinom{r}{v}_q \sum_{u=0}^{r-v} q^{u(u+v-1)} \qbinom{r-v}{u}_q \qbinom{v+j-1}{u+v-1} \\
		& = q^{\binom{v}{2}} \qbinom{r}{v}_q \sum_{u=0}^{r-v} q^{u(u-1)} \qbinom{r-v}{u}_q q^{uv} \qbinom{v+j-1}{j-u} \\
		& = \sum_{u=0}^{r-v} q^{\binom{u}{2}} \qbinom{r}{v}_q \qbinom{r-v}{u}_q q^{\binom{u}{2} + uv + \binom{v}{2}} \qbinom{v+j-1}{j-u} \\
		& = \sum_{u=0}^{r-v} q^{\binom{u}{2}} \qbinom{u+v}{u}_q q^{\binom{u+v}{2}} \qbinom{r}{u+v}_q \qbinom{v+j-1}{j-u}_q \\
	\end{align*}
	where in the first equality we used the well-known $q$-Chu-Vandermonde identity \cite{Andrews-Book-Partitions}*{(3.3.10)} and the other ones are simple algebraic manipulations.
\end{proof}

Combining Proposition~\ref{prop:preliminary-alg-recursion} and Lemma~\ref{lem:chu-vandermonde}, we get the following.

\begin{theorem}
	\label{thm:alg-recursion}
	The expressions $\< \Theta_{e_k} \nabla E_{n-k, r}, e_{n-d} h_d \>$ satisfy the recursion
	\begin{align*}
		\< \Theta_{e_k} \nabla E_{n-k, r}, e_{n-d} h_d \> & = \sum_{j=0}^k t^{n-r-j} \sum_{s=0}^{n-r} \sum_{v=0}^d \sum_{u=0}^{r-v} q^{\binom{u}{2}} \qbinom{u+v}{u}_q q^{\binom{u+v}{2}} \qbinom{r}{u+v}_q \qbinom{v+j-1}{j-u}_q \\
		& \quad \times \qbinom{v+j+s-1}{s}_q \< \Theta_{e_{k-j}} \nabla E_{n-k-r, s}, e_{(n-r-j)-(d-v)} h_{d-v} \> 
	\end{align*}
	with initial conditions $\< \Theta_{e_k} \nabla E_{n-k, r}, e_{n-d} h_d \> = \delta_{r,0} \delta_{k,0} \delta_{d,0}$.
\end{theorem}

Using Theorem~\ref{thm:alg-recursion}, we can get a similar recursion for the other family of polynomials.

\begin{theorem}
	\label{thm:alg-recursion-square}
	The expressions $\< \Theta_{e_k} \nabla \frac{[n-k]_q}{[r]_q} E_{n-k, r}, e_{n-d} h_d \>$ satisfy the recursion
	\begin{align*}
		\< \Theta_{e_k} \nabla \frac{[n-k]_q}{[r]_q} E_{n-k, r}, & e_{n-d} h_d \> = \< \Theta_{e_k} \nabla E_{n-k, r}, e_{n-d} h_d \> \\
		& + \sum_{j=0}^k q^r t^{n-r-j} \sum_{s=0}^{n-r} \sum_{v=0}^d \sum_{u=0}^{r-v} q^{\binom{u}{2}} \qbinom{u+v}{u}_q q^{\binom{u+v}{2}} \qbinom{r-1}{u+v-1}_q \qbinom{v+j}{j-u}_q \\
		& \quad \times \qbinom{v+j+s-1}{s-1}_q \< \Theta_{e_{k-j}} \nabla \frac{[n-k-r]_q}{[s]_q} E_{n-k-r, s}, e_{(n-r-j)-(d-v)} h_{d-v} \> 
	\end{align*}
	with initial conditions $\< \Theta_{e_k} \nabla \frac{[n-k]_q}{[r]_q} E_{n-k, r}, e_{n-d} h_d \> = \delta_{r,0} \delta_{k,0} \delta_{d,0}$.
\end{theorem}

\begin{proof}
	We need to be careful with the initial conditions, as we have the factor $\frac{[n-k]_q}{[r]_q}$ that becomes $\frac{[0]_q}{[0]_q}$ when $n-k=r=0$. However, the property we actually need is \[ \omega(p_{n-k}) = \sum_{r=1}^{n-k} \frac{[n-k]_q}{[r]_q} E_{n-k,r}, \] which only holds for $n-k>0$. For $n-k = 0$ we have $\omega(p_0) = 1$, so we need to set $\frac{[n-k]_q}{[r]_q} E_{n-k,r} = 1$ whenever $n-k=r=0$ for our identities to hold, which satisfies the initial conditions.

	Using Theorem~\ref{thm:alg-recursion}, we have
	\begin{align*}
		\< \Theta_{e_k} \nabla & \frac{[n-k]_q}{[r]_q} E_{n-k, r}, e_{n-d} h_d \> = \frac{[n-k]_q}{[r]_q} \< \Theta_{e_k} \nabla E_{n-k, r}, e_{n-d} h_d \>\\
		%
		& =\left( 1 + q^r \frac{[n-k-r]_q}{[r]_q} \right)\< \Theta_{e_k} \nabla E_{n-k, r}, e_{n-d} h_d \>\\
		%
		& =\< \Theta_{e_k} \nabla E_{n-k, r}, e_{n-d} h_d \> \\
		& \quad + \sum_{j=0}^k q^r t^{n-r-j} \sum_{s=0}^{n-r} \sum_{v=0}^d \sum_{u=0}^{r-v} q^{\binom{u}{2}} \qbinom{u+v}{u}_q q^{\binom{u+v}{2}} \qbinom{r}{u+v}_q \qbinom{v+j-1}{j-u}_q \\
		& \quad \quad \times \qbinom{v+j+s-1}{s}_q \frac{[s]_q}{[r]_q} \< \Theta_{e_{k-j}} \nabla \frac{[n-k-r]_q}{[s]_q} E_{n-k-r, s}, e_{(n-r-j)-(d-v)} h_{d-v} \> \\
		& = \< \Theta_{e_k} \nabla E_{n-k, r}, e_{n-d} h_d \> \\
		& \quad + \sum_{j=0}^k q^r t^{n-r-j} \sum_{s=0}^{n-r} \sum_{v=0}^d \sum_{u=0}^{r-v} q^{\binom{u}{2}} \qbinom{u+v}{u}_q q^{\binom{u+v}{2}} \qbinom{r-1}{u+v-1}_q \qbinom{v+j}{j-u}_q \\
		& \quad \quad \times \qbinom{v+j+s-1}{s-1}_q \< \Theta_{e_{k-j}} \nabla \frac{[n-k-r]_q}{[s]_q} E_{n-k-r, s}, e_{(n-r-j)-(d-v)} h_{d-v} \> 
	\end{align*}
	as desired.
\end{proof}

Once again notice that this implies that $\< \Theta_{e_k} \nabla \frac{[n-k]_q}{[r]_q} E_{n-k, r}, e_{n-d} h_d \>$ is also a polynomial in $\N[q,t]$.

\subsection{Combinatorial recursions}

Let \[ \D_{q,t}(n \backslash r)^{\bullet k, \circ d} \coloneqq \sum_{\pi \in \D(n \backslash r)^{\bullet k, \circ d}} q^{\dinv(\pi)} t^{\area(\pi)} \] and \[ \SQ'_{q,t}(n \backslash r)^{\bullet k, \circ d} \coloneqq \sum_{\pi \in \SQ'(n \backslash r)^{\bullet k, \circ d}} q^{\dinv(\pi)} t^{\area(\pi)} \] be the $q,t$-enumerators of our set of lattice paths. We have the following results.

\begin{theorem}
	\label{thm:combinatorial-recursion}
	The polynomials $\D_{q,t}(n \backslash r)^{\bullet k, \circ d}$ satisfy the recursion
	\begin{align*}
		\D_{q,t}(n \backslash r)^{\bullet k, \circ d}  = \sum_{j=0}^k t^{n-r-j} \sum_{s=0}^{n-r} \sum_{v=0}^d \sum_{u=0}^{r-v} & q^{\binom{u}{2}} \qbinom{u+v}{u}_q q^{\binom{u+v}{2}} \qbinom{r}{u+v}_q \qbinom{v+j-1}{j-u}_q \\
		& \quad \times \qbinom{v+j+s-1}{s}_q \D_{q,t}(n-r-j \backslash s)^{\bullet k-j, \circ d-(r-v)} 
	\end{align*}
	with initial conditions $\D_{q,t}(0 \backslash r)^{\bullet k, \circ d} = \delta_{r,0} \delta_{k,0} \delta_{d,0}$.
\end{theorem}

\begin{proof}
	The initial conditions are trivial, as the only Dyck path of size $0$ has $0$ steps on the main diagonal, $0$ decorated valleys, and $0$ decorated peaks.

	We give an overview of the combinatorial interpretations of all the variables appearing in this formula. We say that a vertical step of a path is \emph{at height $i$} if there are $i$ whole cells in its row between it and the main diagonal.
	
	\begin{itemize}
		\item $r$ is the number of vertical steps at height $0$ that are not decorated valleys.
		\item $j$  is the number of vertical steps at height $0$ that are decorated valleys.
		\item $r-v$ is the number of decorated peaks at height $0$.
		\item $u$ is the number of decorated peaks at height $0$ that are also decorated valleys.
		\item $u+v$ is the number of vertical steps at height $0$ without any kind of decoration.
		\item $s$ is the number of vertical steps at height $1$ that are not decorated valleys.
	\end{itemize}
	
	The recursive step consists of removing all the steps that touch the main diagonal. There are $r+j$ vertical steps that touch the main diagonal, of which $j$ are decorated valleys and $r-v$ are decorated peaks (which are not necessarily disjoint), so after the recursive step we end up with a path in $\D(n-r-j \backslash s)^{\bullet k-j, \circ d-(r-v)}$.

	Let us look at what happens to the statistics of the path. 
	
	The area goes down by the size (i.e.\ $n$) minus the number of vertical steps at height $0$ (i.e.\ $r+j$). This explains the term $t^{n-r-j}$.   
	
	The factor $q^{\binom{u+v}{2}}$ takes into account the primary dinv among all the vertical steps at height $0$ that are neither decorated valleys nor decorated peaks.
	
	The factor $q^{\binom{u}{2}} \qbinom{u+v}{v}_q$ takes into account the primary dinv between all the vertical steps at height $0$ that are neither decorated valleys nor decorated peaks ($u+v$ of them), and all the vertical steps at height $0$ that are both decorated valleys and decorated peaks ($u$ of them), and the expression is explained by the fact that the latter cannot be consecutive (a decorated peak on the main diagonal cannot be followed by a decorated valley).
	
	The factor $\qbinom{r}{u+v}_q$ takes into account the primary dinv between all the vertical steps at height $0$ that are neither decorated valleys nor decorated peaks ($u+v$ of them), and all the vertical steps at height $0$ that are decorated peaks but not decorated valleys ($r-u-v$ of them), which can be interlaced in any possible way.
	
	The factor $\qbinom{v+j-1}{j-u}_q$ takes into account the primary dinv between all the vertical steps at height $0$ that are neither decorated valleys nor decorated peaks ($u+v$ of them), and all the vertical steps at height $0$ that are decorated valleys but not decorated peaks ($j-u$ of them), considering that the first of those must be non-decorated (we cannot start the path with a decorated valley).
	
	Finally, the factor $\qbinom{v+j+s-1}{s}_q$ takes into account the secondary dinv between all the vertical steps at height $1$ that are not decorated valleys ($s$ of them), and all the vertical steps at height $0$ that are not decorated peaks ($v+j$ of them), considering that the first of those must belong to the latter set (we need a vertical step that is not a decorated peak on the main diagonal to go up to the diagonal $y=x+1$).
	
	Summing over all the possible values of $j,s,v$, and $u$, we obtain the stated recursion.
\end{proof}

\begin{theorem}
	\label{thm:combinatorial-recursion-square}
	The polynomials $\SQ'_{q,t}(n \backslash r)^{\bullet k, \circ d}$ satisfy the recursion	
	\begin{align*}
		\SQ'_{q,t}(n \backslash r)^{\bullet k, \circ d} & = \D_{q,t}(n \backslash r)^{\bullet k, \circ d}+ \\
		& \quad + \sum_{j=0}^k q^r t^{n-r-j} \sum_{s=0}^{n-r} \sum_{v=0}^d \sum_{u=0}^{r-v} q^{\binom{u}{2}} \qbinom{u+v}{u}_q q^{\binom{u+v}{2}} \qbinom{r-1}{u+v-1}_q \qbinom{v+j}{j-u}_q \\
		& \quad \quad \times \qbinom{v+j+s-1}{s-1}_q \SQ'_{q,t}(n-r-j \backslash s)^{\bullet k-j, \circ d-(r-v)} 
	\end{align*}
	with initial conditions $\SQ'_{q,t}(0 \backslash r)^{\bullet k, \circ d} = \delta_{r,0} \delta_{k,0} \delta_{d,0}$.
\end{theorem}

\begin{proof}
	The initial conditions are trivial, as the only square path of size $0$ has $0$ steps on the main diagonal, $0$ decorated valleys, and $0$ decorated peaks.

	We give an overview of the combinatorial interpretations of all the variables appearing in this formula. We say that a vertical step of a path is \emph{at height $i$} if there are $i$ whole cells in its row between it and the base diagonal.
	
	\begin{itemize}
		\item $r$ is the number of vertical steps at height $0$ that are not decorated valleys.
		\item $j$  is the number of vertical steps at height $0$ that are decorated valleys.
		\item $r-v$ is the number of decorated peaks at height $0$.
		\item $u$ is the number of decorated peaks at height $0$ that are also decorated valleys.
		\item $u+v$ is the number of vertical steps at height $0$ without any kind of decoration.
		\item $s$ is the number of vertical steps at height $1$ that are not decorated valleys.
	\end{itemize}
	
	The recursive step consists of removing all the steps that touch the base diagonal. We should distinguish whether we start with a Dyck path or a square path: if we start with a Dyck path, then the recursive step is the same as in Theorem~\ref{thm:combinatorial-recursion}, which correspond to the first summand; if we start with a square path that is not a Dyck path, we get the second summand. Since the case of a Dyck path is already dealt with, we only describe the recursion for square paths that are not Dyck paths.
	
	There are $r+j$ vertical steps that touch the base diagonal, of which $j$ are decorated valleys and $r-v$ are decorated peaks (which are not necessarily disjoint), so after the recursive step we end up with a path in $\SQ'(n-r-j \backslash s)^{\bullet k-j, \circ d-(r-v)}$.
	
	Let us look at what happens to the statistics of the path. 
	
	The area goes down by the size (i.e.\ $n$), minus the number of vertical steps at height $0$ (i.e.\ $r+j$). This explains the term $t^{n-r-j}$.
	
	The factor $q^{\binom{u+v}{2}}$ takes into account the primary dinv among all the vertical steps at height $0$ that are neither decorated valleys nor decorated peaks.
	
	The factor $q^{\binom{u}{2}} \qbinom{u+v}{v}_q$ takes into account the primary dinv between all the vertical steps at height $0$ that are neither decorated valleys nor decorated peaks ($u+v$ of them), and all the vertical steps at height $0$ that are both decorated valleys and decorated peaks ($u$ of them), and the expression is explained by the fact that the latter cannot be consecutive (a decorated peak on the main diagonal cannot be followed by a decorated valley).
	
	The factor $\qbinom{r-1}{u+v-1}_q$ takes into account the primary dinv between all the vertical steps at height $0$ that are neither decorated valleys nor decorated peaks ($u+v$ of them), and all the vertical steps at height $0$ that are decorated peaks but not decorated valleys ($r-u-v$ of them), and the last of these must be a step without decorations: a decorated peak cannot be followed by a decorated valley, and the last vertical step at height $0$ cannot be a decorated peak, since the shift of the path is positive and it has to finish above the main diagonal. This explains the $r-1$ and the $u+v-1$. 
	
	The factor $\qbinom{v+j}{j-u}_q$ takes into account the primary dinv between all the vertical steps at height $0$ that are neither decorated valleys nor decorated peaks ($u+v$ of them), and all the vertical steps at height $0$ that are decorated valleys but not decorated peaks ($j-u$ of them); unlike the previous case, now the interlacing can be any.
	
	Finally, the factor $\qbinom{v+j+s-1}{s-1}_q$ takes into account the secondary dinv between all the vertical steps at height $1$ that are not decorated valleys($s$ of them), and all the vertical steps at height $0$ that are not decorated peaks ($v+j$ of them), and since the shift of the path is positive and it has to finish above the main diagonal, the last vertical step at height $1$ must occurr after the last vertical step at height $0$.
	
	Summing over all the possible values of $j,s,v$, and $u$, we obtain the stated recursion.
\end{proof}

\subsection{The main theorems}

We are now ready to prove Theorem~\ref{thm:valley-schroeder} and Theorem~\ref{thm:valley-square-schroeder}.

\begin{theorem}
	\begin{equation} \label{eq:SchroValleyEnk}
		\< \Theta_{e_k} \nabla E_{n-k, r}, e_{n-d} h_d \> = \sum_{\pi \in \D(n \backslash r)^{\bullet k, \circ d}} q^{\dinv(\pi)} t^{\area(\pi)}. 
	\end{equation}
\end{theorem}
\begin{proof}
	It follows immediately by combining Theorem~\ref{thm:alg-recursion} and Theorem~\ref{thm:combinatorial-recursion}.
\end{proof}
\begin{theorem}
\begin{equation} \label{eq:SchroSquareValleyEnk} \< \Theta_{e_k} \nabla \frac{[n-k]_q}{[r]_q} E_{n-k, r}, e_{n-d} h_d \> = \sum_{\pi \in \SQ'(n \backslash r)^{\bullet k, \circ d}} q^{\dinv(\pi)} t^{\area(\pi)}. 
\end{equation}	
\end{theorem}
\begin{proof}
Use \eqref{eq:SchroValleyEnk} together with Theorem~\ref{thm:alg-recursion-square} and Theorem~\ref{thm:combinatorial-recursion-square}.
\end{proof}

\begin{proof}[Proof of Theorem~\ref{thm:valley-schroeder} and Theorem~\ref{thm:valley-square-schroeder}]

Take the sum over $r$ of both \eqref{eq:SchroValleyEnk} and \eqref{eq:SchroSquareValleyEnk}.
\end{proof}

\section{The Catalan case of the extended valley Delta conjecture}

In \cite{Iraci-VandenWyngaerd-pushing-2021}, the authors show that the valley version of the Delta conjecture implies the extended version of the same conjecture. The argument requires the conjecture to hold in full generality to function, but we can still recycle it to derive the Catalan case of the extended conjecture from the Schr\"oder case of the original one.

We need the following result, suggested by a combinatorial argument by the second author and Vanden Wyngaerd, and then proved by the first author and Romero.

\begin{theorem}[{\cite{DAdderio-Romero-Theta-Identities-2020}*{Corollary~9.2}}]
	Given $d,n,k,r \in \N$, we have
	\begin{align*}
		h_d^\perp & \Theta_{e_k} \nabla E_{n-k,r} = \sum_{p=0}^d t^{d-p} \sum_{i=0}^p q^{\binom{i}{2}} \qbinom{r-p+i}{i}_q \qbinom{r}{p-i}_q \Delta_{h_{d-p}} \Theta_{e_{k-i}} \nabla E_{n-d-(k-i), r-p+i}.
	\end{align*}
\end{theorem}

Taking the scalar product with $e_{n-d}$, we get the following.

\begin{proposition}
	\label{prop:sf-identity}
	Given $d,n,k,r \in \N$, we have
	\begin{align}
	\label{eq:RecSF}	\< \Theta_{e_k} \nabla E_{n-k,r}, e_{n-d} h_d \> & = \sum_{p=0}^m t^{m-p} \sum_{i=0}^p q^{\binom{i}{2}} \qbinom{r-p+i}{i}_q \qbinom{r}{p-i}_q \\
	\notag	& \quad \times \< \Delta_{h_{d-p}} \Theta_{e_{k-i}} \nabla E_{n-d-(k-i), r-p+i}, e_{n-d} \>.
	\end{align}
\end{proposition}

We want to prove that the $q,t$-enumerators of the corresponding sets satisfy the same relations.

\begin{theorem}
	\label{thm:comb-identity}
	Given $d,n,k,r \in \N$, we have
	\begin{align}
		\label{eq:RecComb}\D_{q,t}(n \backslash r)^{\bullet k, \circ d} = \sum_{p=0}^d t^{d-p} \sum_{i=0}^p q^{\binom{i}{2}} \qbinom{r-p+i}{i}_q \qbinom{r}{p-i}_q \D_{q,t}(d-p,n-d \backslash r-p+i)^{\bullet k-i}.
	\end{align}
\end{theorem}

\begin{proof}
	Using the same idea as in \cite{Iraci-VandenWyngaerd-pushing-2021}*{Theorem~5.1}, we delete the decorated peaks on the main diagonal, and apply the \emph{pushing algorithm} to the remaining decorated peaks, that is, we swap the horizontal and the vertical step they are composed of, becoming valleys. If the peak is also a decorated valley, it becomes a decorated valley. If $p$ is the number of decorated peaks on the main diagonal, and $i$ the number of those that also are decorated valleys, we have that the loss of area given by the pushing contributes for a factor $t^{d-p}$, removing the decorated peaks from the main diagonal contributes for a factor \[ q^{\binom{i}{2}} \qbinom{r-p+i}{i}_q \qbinom{r}{p-i}_q, \] and we are left with a path in $\D(d-p,n-d \backslash r-p+i)^{\bullet k-i}$: see \cite{Iraci-VandenWyngaerd-pushing-2021}*{Theorem~5.1} for more details. The thesis follows.
\end{proof}

Combining the two statements we get the following.

\begin{proposition}
	Given $d,n,k,r \in \N$, we have
	\begin{equation} \label{eq:SchroExtValleyEnk}
		\< \Delta_{h_d} \Theta_{e_k} \nabla E_{n-k,r}, e_{n} \> = \sum_{\pi \in \D(d, n \backslash r)^{\bullet k}} q^{\dinv(\pi)} t^{\area(\pi)}. 
	\end{equation}
\end{proposition}
\begin{proof}
We can rewrite \eqref{eq:RecSF} as
\begin{align*}
	t^d \< \Delta_{h_{d}} & \Theta_{e_{k}} \nabla E_{n-d-k, r}, e_{n-d} \> = \< \Theta_{e_k} \nabla E_{n-k,r}, e_{n-d} h_d \>\\ & - \sum_{p=1}^d t^{d-p} \sum_{i=0}^p q^{\binom{i}{2}} \qbinom{r-p+i}{i}_q \qbinom{r}{p-i}_q \< \Delta_{h_{d-p}} \Theta_{e_{k-i}} \nabla E_{n-d-(k-i), r-p+i}, e_{n-d} \>.
\end{align*}
and \eqref{eq:RecComb} as
	\begin{align*}
		t^d \D_{q,t}(d,n-d \backslash r)^{\bullet k} & = \D_{q,t}(n \backslash r)^{\bullet k, \circ d} \\ 
	 & \quad - \sum_{p=1}^d t^{d-p} \sum_{i=0}^p q^{\binom{i}{2}} \qbinom{r-p+i}{i}_q \qbinom{r}{p-i}_q \D_{q,t}(d-p,n-d \backslash r-p+i)^{\bullet k-i}.
\end{align*}
Using \eqref{eq:SchroValleyEnk} and induction on $d$ (the base case $d=0$ is simply \eqref{eq:SchroValleyEnk}) we see that the right hand sides are equal, hence so are the left hand sides. This completes the proof.
\end{proof}

Taking the sum over $r$ in \eqref{eq:SchroExtValleyEnk}, we get the desired result.

\begin{theorem}
	Given $d,n,k \in \N$, we have
	\[ \< \Delta_{h_d} \Theta_{e_k} \nabla e_{n-k}, e_{n} \> = \sum_{\pi \in \D(d, n)^{\bullet k}} q^{\dinv(\pi)} t^{\area(\pi)}. \]
\end{theorem}

\section{The Catalan case of the extended valley square conjecture}

Next, we want to prove the analogous statement for the square version. In this case, we need to assume that the combinatorial side of (a refinement of) Conjecture~\ref{conj:valleyDelta} is a symmetric function (cf.\ Remark~\ref{rmk:SFcombside}).

Indeed, the argument in \cite{Iraci-VandenWyngaerd-Valley-Square-2021} can be recycled to show that, if the Catalan case of the valley version of the extended Delta conjecture holds, then the same case of the square version of the same conjecture also holds.

We need the following result.
\begin{proposition}[{\cite{Iraci-VandenWyngaerd-Valley-Square-2021}*{Corollary~3}}]
	\label{cor:square-to-dyck}
	Let 
	\begin{equation} \label{eq:LD_comb}
	\LD_{q,t;x}(d, n \backslash r)^{\bullet k} = \sum_{\pi \in \LD(d, n)^{\bullet k}} q^{\dinv(\pi)} t^{\area(\pi)} x^\pi 
	\end{equation} 
	and 
	\begin{equation}  \label{eq:LSQ_comb}
	\LSQ'_{q,t;x}(d, n \backslash r)^{\bullet k} = \sum_{\pi \in \LSQ'(d, n)^{\bullet k}} q^{\dinv(\pi)} t^{\area(\pi)} x^\pi. 
	\end{equation} 
	Then 
	\[ \LSQ'_{q,t;x}(d, n \backslash r)^{\bullet k} = \frac{[n-k]_q}{[r]_q} \LD_{q,t;x}(d, n \backslash r)^{\bullet k}. \]
\end{proposition}

Now, assuming that \eqref{eq:LD_comb} (and hence also \eqref{eq:LSQ_comb}) is a symmetric function, again by the theory of shuffles, taking the scalar product with $e_{n}$ isolates the subsets of the paths whose reading word is $1, 2, \dots, n$, which are in statistic-preserving bijection with unlabelled paths by simply removing the positive labels (there is a unique way to put them back once the reading word is fixed). The following theorem is now immediate.

\begin{theorem}
	Given $d,n,k,r \in \N$, if \eqref{eq:LD_comb} is a symmetric function, then
	\[ \< \Delta_{h_d} \Theta_{e_k} \nabla \omega(p_{n-k}), e_{n} \> = \sum_{\pi \in \SQ'(d, n)^{\bullet k}} q^{\dinv(\pi)} t^{\area(\pi)}. \] 
\end{theorem}

\begin{proof}
	Using \eqref{eq:SchroExtValleyEnk} and Proposition~\ref{cor:square-to-dyck}, we have
	\begin{align*}
		\< \Delta_{h_d} \Theta_{e_k} \nabla \frac{[n-k]_q}{[r]_q} E_{n-k, r}, e_{n} \> & = \frac{[n-k]_q}{[r]_q} \sum_{\pi \in \D(d, n\backslash r)^{\bullet k}} q^{\dinv(\pi)} t^{\area(\pi)} \\
		& = \left\< \frac{[n-k]_q}{[r]_q} \LD_{q,t;x}(d, n \backslash r)^{\bullet k}, e_n \right\> \\
		& = \< \LSQ'_{q,t;x}(d, n \backslash r)^{\bullet k}, e_n \> \\
		& = \sum_{\pi \in \SQ'(d, n \backslash r)^{\bullet k}} q^{\dinv(\pi)} t^{\area(\pi)},
	\end{align*}
	so by taking the sum over $r$ the thesis follows.
%
%
\end{proof}

\section{Future directions}

In \cite{DAdderio-Iraci-VandenWyngaerd-GenDeltaSchroeder-2019}, the authors gave an algebraic recursion for the (conjectural) $q,t$-enumerators of the Schr\"oder case of the extended valley Delta conjecture. Hence it would be enough to find a combinatorial recursion that matches the algebraic one to prove that case as well; however, we were not able to do so.

In \cite{DAdderio-Iraci-newdinv-2019}, a combinatorial recursion for the ``ehh'' case of the shuffle theorem is given, and it would be interesting to extend such recursion to the valley-decorated version of the combinatorial objects, which in turn would lead to a proof of the Schr\"oder case of the extended valley Delta conjecture via the pushing algorithm in \cite{Iraci-VandenWyngaerd-Valley-Square-2021}; once again, our attempts were unfruitful.

Finally, as the classical recursion for the $q,t$-Catalan \cite{Garsia-Haglund-qtCatalan-2002} is an iterated version of the compositional one \cite{Haglund-Morse-Zabrocki-2012}, and both these recursions extend to the rise version of the Delta conjecture \cites{Zabrocki-4Catalan-2016, DAdderio-Mellit-Compositional-Delta-2020}, it might be the case that the same phenomenon occurrs here, and that there is a compositional refinement of the recursion for $d=0$ that might lead to a full proof of the conjecture; at the moment we do not know what should such a refinement look like.

\bibliographystyle{amsalpha}
\bibliography{bibliography}

\end{document}